\documentclass[12pt]{article}
\usepackage{amsmath,amssymb}
\usepackage[T1]{fontenc}
\usepackage{latexsym}
\usepackage{pdfsync}
\usepackage{epsfig}
\usepackage{proof}
\usepackage[utf8]{inputenc}
\usepackage[a4paper, top=3cm, bottom=3cm ]{geometry}
\usepackage{graphicx}
\newtheorem{thm}{Theorem}[section]

\newtheorem{lem}{Lemma}[section]

\title
{$L^p$-norm estimate for the Bergman projection on  Hartogs triangle}

\author{\normalsize Tomasz Beberok \\
\small Faculty of Mathematics and Computer Science, Jagiellonian University,\\
\small Lojasiewicza 6, 30-048 Krakow, Poland \\}

\date{}

\begin{document}

\begin{center}
  \textbf{On $L^p$ Markov type inequality for some cuspidal domains in $\mathbb{R}^2$
}
\end{center}
\vskip1em
\begin{center}
  Tomasz Beberok
\end{center}

\vskip2em

\noindent \textbf{Abstract.} The purpose of this paper is to study a Markov type inequality for algebraic polynomials in $L^p$ norm on two-dimensional cuspidal domains.
\vskip1em


\section{Introduction}
\label{} In the space $\mathbb{R}^d$ we consider the Euclidean norm: $|\mathbf{x}|:=\sqrt{|x_1|^2 + \cdots |x_d|^2}$, where
$x = (x_1, \ldots, x_d)$. For a nonempty compact set $E \subset \mathbb{R}^n$, $1 \leq p < \infty$ and $h: E \rightarrow \mathbb{R}$  for which the $p$th power of the absolute value is Lebesgue integrable, we put
$$\|h\|_{L^p(E)}:=\left( \int_E |h(x)|^p \, dx \right)^{1/p}.$$
 If two sequences $z_n$ and $w_n$ of real numbers have the property that $w_n \neq 0$ and the sequence $|z_n|/|w_n|$ has finite positive limit as $n \rightarrow \infty$, we write $z_n \sim w_n$. Throughout the paper, $\mathcal{P}_n(\mathbb{R}^d)$ denotes the space of real algebraic polynomials of $d$ variables and degree at most $n$ and $P_n^{(\alpha,\beta)}$ denotes the Jacobi  polynomial of degree $n$ associated to parameters $\alpha, \beta$. Moreover, $\mathbb{N}:=\{1,2,3, \ldots\}$ and $\mathbb{N}_0:= \mathbb{N} \cup \{0\}$.
\vskip0.3cm
\noindent \textbf{Definition.} Let $1 \leq p < \infty$. We say that a compact set $\emptyset \neq E \subset \mathbb{R}^d$ satisfies \textbf{ $L^p$ Markov type inequality} (or: is a \textbf{$L^p$ Markov set}) if there exist $\kappa,C > 0$ such that, for each polynomial $P \in \mathcal{P}_n(\mathbb{R}^d)$ and each $1 \leq j \leq d$,
\begin{align}\label{Markov}
  \left\|\frac{\partial P}{ \partial x_j} \right\|_{L^p(E)}  \leq C n^{\kappa} \|P\|_{L^p(E)}.
\end{align}
We denote by $B(\mathbf{a},r) \subset \mathbb{R}^d$ the closed Euclidean ball with center $\mathbf{a}$ and radius $r$, and $\mathbb{S}^{d-1} = \{\mathbf{x} \in \mathbb{R}^d : |\mathbf{x}| = 1\}$ is the unit sphere. For any $r > 0$, $\mathbf{a} \in  \mathbb{R}^d$ and $\mathbf{u} \in \mathbb{S}^{d-1}$ the cylinder $L_{\mathbf{a}}(r, \mathbf{u})$ with center $\mathbf{a}$, radius $r > 0$,  and axis $\mathbf{u}$ is given
by
$$L_{\mathbf{a}}(r,\mathbf{u}):=\{ \mathbf{x} \in \mathbb{R}^d : |\mathbf{x}-\mathbf{a}|^2 < r^2 + \langle  \mathbf{x}-\mathbf{a}, \mathbf{u} \rangle^2 \}.$$
Furthermore, $l_{\mathbf{x}}(\mathbf{u})$ will denote the line in $\mathbb{R}^d$ in direction $\mathbf{u} \in \mathbb{S}^{d-1}$ through point $\mathbf{x} \in \mathbb{R}^d$.
\newline \indent  Following Kro\'{o} \cite{K2}, we introduce a graph domain with respect to the cylinder $L_{\mathbf{a}}(r,\mathbf{u})$ and a piecewise graph domain.
  \vskip0.3cm
\noindent \textbf{Definition.} $K$ is called a \textbf{graph domain} with respect to the cylinder $L_{\mathbf{a}}(r,\mathbf{u})$ if for every $\mathbf{x} \in B(\mathbf{a},r)$ we have that $l_{\mathbf{x}}(\mathbf{u}) \cap K = [A_1(\mathbf{x}),A_2(\mathbf{x})]$ with $A_i(\mathbf{x})$, $i=1,2$ being continuous for $\mathbf{x} \in B(\mathbf{a},r)$ and
$$\delta_r (\mathbf{a},\mathbf{u}):= \inf_{\mathbf{x} \in B(\mathbf{a},r)}  |A_1(\mathbf{x})- A_2(\mathbf{x})| >0.$$
Moreover, $K \subset \mathbb{R}^d$ is a \textbf{piecewise graph domain} if it can be covered by finite number of cylinders so that $K$ is a graph domain with respect to each of them. \newline \indent
Similarly to \cite{K2} $\omega_K(\cdot)$ denotes the modulus of continuity of the boundary of piecewise graph domain $K$ which is defined as the maximum of modula of continuity of all functions $A_i (\cdot)$ involved in the corresponding finite covering by cylinders. If $\epsilon:=\epsilon_n(K)$ is a solution of the equation
$$2n^2 \omega_K\left( \frac{\epsilon}{n^2} \right)=1, \quad n \in \mathbb{N},$$ then the main result of the mentioned paper of Kro\'{o} is

\begin{thm}\label{KrooThm}
  Let $K \subset \mathbb{R}^d$ be a cuspidal piecewise graph domain. Then there exists a positive constant $B$, depending on $K$ and on $p$, such that for $Q \in \mathcal{P}_n(\mathbb{R}^d)$, $n \in \mathbb{N}$,
  \begin{align*}
    \| \nabla Q \|_{L^p(K)} \leq B \frac{n^2}{\epsilon_n} \| Q \|_{L^p(K)}.
  \end{align*}
Here $\nabla Q := \max_{1\leq j \leq d} \left| \frac{\partial Q}{ \partial x_j} \right|$. In particular, if $K$ is $Lip \gamma$, $0< \gamma <1$ then
  \begin{align*}
    \| \nabla Q \|_{L^p(K)} \leq B n^{\frac{2}{\gamma}} \| Q \|_{L^p(K)}
  \end{align*}
and the exponent $\frac{2}{\gamma}$ is best possible.
\end{thm}
\noindent The above theorem is a particular result in the general problem of estimating the exponent of the growth rate (with respect to the
degree $n$) of the best comparability constant of the semi-norm  $\| \nabla \cdot \|_{L^p(\Omega)}$
and the norm $\| \cdot \|_{L^p(\Omega)}$ acting on the space $\mathcal{P}_n(\mathbb{R}^d)$ for a given compact set $\Omega$. More precisely, the \textbf{Markov exponent in $L^p$-norm of a $L^p$ Markov set $K$} is defined as the infimum of $l$ as $l$ ranges over all positive numbers such that there exists a constant $C>0$, independent of $n$, with the property that $L^p$ Markov type inequality (\ref{Markov}) holds (with $l$ and $C$), which we denote by~$\mu_p(K)$.
\newline  \indent The notion of Markov exponent (in the supremum norm) appears first in \cite{BP}. The Markov exponent has many interesting applications in approximation theory, constructive function theory and in analysis (for instance, to Sobolev inequalities or Whitney-type extension problems see \cite{BOS}, \cite{PP} and \cite{P1}). It is known that $\mu_p(K) \geq 2$ for every compact subset $E$ of $\mathbb{R}^d$. In \cite{G2} it is proved that if $K$ is a locally Lipschitzian compact subsets of $\mathbb{R}^d$, then $\mu_p(K)=2$. See also \cite{MB2}, \cite{G3},  \cite{G5} and \cite{SMN}. In the case of cuspidal domains, see \cite{K1} and \cite{KS}.   Markov’s inequality and its various generalizations were studied in a large number of papers, it is beyond the scope of this paper to give a complete bibliography, an extensive survey of the results is given in \cite{BE}, \cite{MM}, \cite{RS} and \cite{P2}.
\newline \indent
One of the purposes of this note is to show that, if $d=2$, the factor $\frac{n^2}{\epsilon_n}$ is best possible for larger class of domains then $Lip \gamma$. Another goal is to prove that for every sequence $\{ \epsilon_n \}$, satisfying certain properties, there exist a compact set $D \subset \mathbb{R}^2$, a constant $M>0$ and a sequence of polynomials $P_n$ such that
\begin{align*}
  \frac{ \| \nabla P_n \|_{L^p(D)} } { \| P_n \|_{L^p(D)} }  \sim \frac{n^2}{\epsilon_n} \quad \text{and} \quad \| \nabla Q \|_{L^p(D)} \leq M \frac{n^2}{\epsilon_n}  \| Q \|_{L^p(D)}
\end{align*}
for any real algebraic polynomial $Q$ of two variables and degree at most $n$. Moreover, for every $\iota \geq 1$ and $1\leq p < \infty$ we give an example of connected compact subset $E_\iota$ of $\mathbb{R}^2$ such that $\mu_p(E_\iota)=2 \iota$ and the inequality (\ref{Markov}) does not hold with the exponent $\mu_p(E_\iota)$.

\section{A sharp Markov type inequality}
\noindent \textbf{Definition.} Let $f: (a,b) \rightarrow \mathbb{R}$ be a convex function. The \textbf{index of convexity} of $f$ is defined by \begin{align*}
             I_{conv}(f):= \sup \{r \geq 1 : (f)^{\frac{1}{r}} \text{ is convex}\}.
           \end{align*}
            \indent
For a given point $\mathbf{a} \in \mathbb{R}^2$ and a line $l_{\mathbf{b}}(\mathbf{u}) \subset \mathbb{R}^2$, $S_{l_{\mathbf{b}}(\mathbf{u}) }(\mathbf{a})$ stands for the point that is symmetric to the point $\mathbf{a}$ with respect to the line $l_{\mathbf{b}}(\mathbf{u})$. The point $\pi_{l_{\mathbf{b}}(\mathbf{u})} (\mathbf{a}) \in \mathbb{R}^2$ is the orthogonal projection of the point $\mathbf{a}$ onto the line $l_{\mathbf{b}}(\mathbf{u})$ i.e., $\pi_{l_{\mathbf{b}}(\mathbf{u})}(\mathbf{a}):=l_{\mathbf{b}}(\mathbf{u}) \cap l_{\mathbf{a}}(\mathbf{w}),$ where $\mathbf{w} \perp \mathbf{u}$. \newline
\indent
Let $K \subset \mathbb{R}^2$ be a piecewise graph domain. Suppose that $\mathbf{z} \in K$ is one of the strongest cuspidal point of $K$ i.e., there exists a cylinder $L_{\mathbf{a}}(r,\mathbf{u})$ such that $K$ is a graph domain with respect to it, $\mathbf{z}= A_2(\mathbf{b})$ for some $\mathbf{b} \in B(\mathbf{a},r)$ and for all sufficiently large $n$, $\omega_K(\epsilon_n/n^2)=|\mathbf{z}-A_2(\mathbf{g_n})|$ for some $\mathbf{g_n} \in B(\mathbf{a},r)$. Let $\mathbf{w} \in \mathbb{S}^{1},$ $\mathbf{w} \perp \mathbf{u}$. We say that $\mathbf{z}$ is regular if there exist $\mathbf{o} \in l_{\mathbf{b}}(\mathbf{u})$, and a function $f:[\mathbf{o},\mathbf{z}] \rightarrow \mathbb{R}^2$ such that $f(\mathbf{z})=\mathbf{z}$,
\begin{align*}
  &[f(\mathbf{x}), \pi_{l_{\mathbf{z}}(\mathbf{u})}(f(\mathbf{x})] \subset K \cap l_{\mathbf{x}}(\mathbf{w}) \subset [f(\mathbf{x}),S_{l_{\mathbf{z}}(\mathbf{u})} (f(\mathbf{x}))] \quad \text{for all } \mathbf{x} \in [\mathbf{o},\mathbf{z}], \\
  &d(t):= \text{dist}(l_{f((1-t)\mathbf{o} + t\mathbf{z})} (\mathbf{u}),  l_{\mathbf{z}} (\mathbf{u})  ) \quad  \text{is convex  on the interval }  (0,1) \,\,\, \text{and} \\
  & I_{conv}(d) < \infty.
\end{align*}

\begin{thm}\label{mainKroo}
  Let $K \subset \mathbb{R}^2$ be a piecewise graph domain. Suppose that one of the strongest cuspidal point of $K$  is regular.  If $\epsilon_n$ is a solution of the equation $2n^2 \omega_K\left( \frac{\epsilon_n}{n^2} \right)=1$, $n \in \mathbb{N}$, then there exist $\Psi>0$ and a sequence of polynomials $P_n \in \mathcal{P}_n(\mathbb{R}^2)$ such that
   \begin{align*}
     \| \nabla P_n \|_{L^p(K)} \geq \Psi \frac{n^2}{\epsilon_n}  \| P_n \|_{L^p(K)}.
   \end{align*}
\end{thm}

\begin{proof}
Without loss of generality we may suppose that $K \subset [0,1] \times [-1,1]$, $\mathbf{v}=(1,0)$ is the strongest cuspidal point of $K$ and $\{(x,y) \in \mathbb{R}^2 : \eta \leq x \leq 1, 0 \leq y \leq f(x)\} \subset K \subset [0,\eta] \times [-1,1] \cup \{(x,y) \in \mathbb{R}^2 : \eta \leq x \leq 1, -f(x) \leq y \leq f(x)\}$ for some  $0< \eta <1$ and a convex function $f: [\eta,1] \rightarrow \mathbb{R}$ with the property that $f(1)=0$ and $I_{conv}(f) < \infty$. (This can be achieved by shifting the point $\mathbf{o}$ into the origin,  rotating around the origin  and dilating the space by a proper constant.)
 Let $x_n=f^{-1}(\epsilon_n/n^2)$ for all sufficiently large $n$.  Then
\begin{align}\label{gen}
  \iint_K |y P^{(\alpha,\beta)}_n(x)|^p \,dxdy \leq& \int_{0}^{\eta} \int_{-1}^{1} |y P^{(\alpha,\beta)}_n(x)|^p \,dydx \nonumber \\&+ \int_{\eta}^{x_n}  \int_{-f(x)}^{f(x)} |y P^{(\alpha,\beta)}_n(x)|^p \,dydx \nonumber \\&+ \int_{x_n}^{1}  \int_{-f(x)}^{f(x)} |y P^{(\alpha,\beta)}_n(x)|^p \,dydx
\end{align}
\indent Our plan is to obtain the estimates of each integral on the right side. We start with the last one. It is clear that
\begin{align}\label{last}
  \int_{x_n}^{1} \int_{-f(x)}^{f(x)} |y P^{(\alpha,\beta)}_n(x)|^p \,dxdy =& \frac{2}{p+1}  \int_{x_n}^{1} (f(x))^{p+1} |P^{(\alpha,\beta)}_n(x)|^p \,dx \nonumber \\ &\leq \frac{2}{p+1} \left(\frac{\epsilon_n}{n^2} \right)^{p+1}  \int_{x_n}^{1}  |P^{(\alpha,\beta)}_n(x)|^p \,dx.
\end{align}
Then the change of variable $x=\cos \theta$ gives us
\begin{align*}
   \int_{x_n}^{1}  |P^{(\alpha,\beta)}_n(x)|^p \,dx = \int_{0}^{u_n}  |P^{(\alpha,\beta)}_n(\cos \theta)|^p \sin \theta \,d\theta.
\end{align*}
Here $u_n=\arccos x_n$. Since $\mathbf{v}=(1,0)$ is the strongest cuspidal point of $K$, it follows that  $x_n=1-\sqrt{\frac{1}{4n^4}-\frac{\epsilon_n^2}{n^4}}$. By the fact that $\epsilon_n \rightarrow 0$ there exists a natural number $n_0$ such that $1-\frac{1}{2n^2}\leq x_n \leq 1- \frac{1}{4n^2}$ for all $n \geq n_0$. Hence there exist a natural number $n_1$ and  positive constants $a,b$ such that $\frac{a}{n} \leq u_n \leq \frac{b}{n}$ for all $n \geq n_1$. Applying certain properties of Jacobi polynomials $P^{(\alpha,\beta)}_n(x)$ verified in [8], $(7.32.5)$, p. 169, we conclude that there exists a natural number $n_2$ so that
\begin{align}\label{l2}
  \int_{0}^{u_n}  |P^{(\alpha,\beta)}_n(\cos \theta)|^p \sin \theta \,d\theta \leq C n^{\alpha p} \int_{0}^{u_n} \theta \,d\theta \leq \frac{Cb^2}{2} n^{\alpha p-2}
\end{align}
for $n \geq n_2$ and appropriately adjusted constant $C$. Then by (\ref{last}) and (\ref{l2})
\begin{align}\label{l3}
  \int_{x_n}^{1} \int_{-f(x)}^{f(x)} |y P^{(\alpha,\beta)}_n(x)|^p \,dydx \leq \frac{Cb^2}{p+1} \epsilon_n^{p+1} n^{\alpha p-2p-4}
\end{align}
for all sufficiently large $n$. \newline \indent Now select $\alpha > -1$ such that $\alpha p + p/2 -2 > 2I_{conv}(f)(p+1)$. It is easy to see that
\begin{align}\label{mid}
  \int_{\eta}^{x_n}  \int_{-f(x)}^{f(x)} |y P^{(\alpha,\beta)}_n(x)|^p \,dydx= \frac{2}{p+1}  \int_{\eta}^{x_n} (f(x))^{p+1} |P^{(\alpha,\beta)}_n(x)|^p \,dx.
\end{align}
Let $\sigma = \arccos \eta$. Proceeding similarly as before, we obtain
\begin{align}\label{mid2}
  \ \frac{2}{p+1}  \int_{\eta}^{x_n} (f(x))^{p+1} |P^{(\alpha,\beta)}_n(x)|^p \,dx \leq \frac{2 \Lambda n^{-p/2}}{(p+1)} \int_{u_n}^{\sigma } (f(\cos \theta))^{p+1}  \theta^{-\alpha p - p/2 } \sin \theta \,d\theta
\end{align}
for appropriately adjusted constant $\Lambda$ and all sufficiently large $n$. Since $\sin x \leq x$ for $x \geq 0$, we have
\begin{align}\label{mid2a}
  \ \frac{2}{p+1}  \int_{\eta}^{x_n} (f(x))^{p+1} |P^{(\alpha,\beta)}_n(x)|^p \,dx \leq \frac{2\Lambda n^{-p/2}}{(p+1)} \int_{u_n}^{ \sigma } (f(\cos \theta))^{p+1}  \theta^{-\alpha p - p/2 +1} \,d\theta.
\end{align}
Integration by parts gives us
\begin{align}\label{IbyP}
   \int_{u_n}^{\sigma } (f(\cos \theta))^{p+1}  \theta^{-\alpha p - p/2 +1} \,d\theta=&\left[ \frac{(f(\cos \theta))^{p+1}  \theta^{-\alpha p - p/2 +2}}{-\alpha p - p/2 +2}\right]^{ \sigma}_{u_n} \nonumber \\ &+ \int_{u_n}^{ \sigma } \frac{ (p+1)(f(\cos \theta))^{p}  }{(-\alpha p - p/2 +2) \theta^{\alpha p + p/2 -2}} f'(\cos \theta) \sin \theta \,d\theta.
\end{align}
If $-1\leq x \leq 1$, then $\sqrt{1-x^2} \arccos x \leq 2(1-x)$. Hence
   \begin{align}\label{derE}
     - \lambda f'(\cos \lambda) \sin \lambda  \leq     -2(1-\cos \lambda)f'(\cos \lambda)
   \end{align}
whenever $\lambda \in (0,\sigma]$. From the definition of index of convexity of $f$ it follows that for each fixed $\delta >0$,
$(f)^{\frac{1}{I_{conv}(f) + \delta }}$ is concave. Hence
\begin{align*}
  (f(x))^{ \frac{1}{I_{conv}(f) + \delta }  } \geq - (f(x))^{ \frac{1}{I_{conv}(f) + \delta } -1 }  \frac{f'(x)(1-x)}{I_{conv}(f) + \delta }
\end{align*}
for any $x \in [\eta,1)$. Therefore
 \begin{align}\label{appindex}
   (I_{conv}(f) + \delta) f(x) \geq -f'(x)(1-x)
 \end{align}
for all $x\in [\eta,1]$. Then, by (\ref{derE}) and (\ref{appindex}),
\begin{align}\label{estderE}
  \int_{u_n}^{ \sigma } \frac{ (p+1)(f(\cos \theta))^{p}  \theta^{-\alpha p - p/2 +2}}{-\alpha p - p/2 +2} & f'(\cos \theta) \sin \theta \,d\theta \nonumber \\ & \leq \int_{u_n}^{ \sigma } \frac{ 2(p+1) (I_{conv}(f) + \delta) (f(\cos \theta))^{p+1} }{(\alpha p + p/2 - 2) \theta^{\alpha p + p/2 -1}}  \,d\theta .
\end{align}
Thus, by (\ref{IbyP}) and (\ref{estderE}),
\begin{align*}
  \int_{u_n}^{\sigma } (f(\cos \theta))^{p+1}  \theta^{-\alpha p - p/2 +1} \,d\theta &\leq \frac{(f(\cos u_n))^{p+1}  u_n^{-\alpha p - p/2 +2}}{\alpha p + p/2 - 2 - 2(p+1) (I_{conv}(f) + \delta)} \\ & \leq \frac{ \epsilon_n^{p+1}  a^{-\alpha p - p/2 +2} n^{\alpha p + p/2 -2p -4} }{\alpha p + p/2 - 2 - 2(p+1) (I_{conv}(f) + \delta)}
\end{align*}
whenever $\alpha p + p/2 -2 > 2(p+1) (I_{conv}(f) + \delta)$. Together with (\ref{mid}), (\ref{mid2}) and (\ref{mid2a}), this last estimate implies that for every $\alpha > -1$ such that $\alpha p + p/2 -2 > 2(p+1) I_{conv}(f)$ there exists a constant $C_1>0$, independent of $n$, with
 \begin{align}\label{midf}
  \int_{\eta}^{x_n}  \int_{-f(x)}^{f(x)} |y P^{(\alpha,\beta)}_n(x)|^p \,dydx \leq C_1 \epsilon_n^{p+1} n^{\alpha p-2p-4}.
\end{align}
 \newline \indent It now remains to prove that there exists a positive constant $C_2$, independent of $n$, such that
\begin{align*}
 \int_{0}^{\eta} \int_{-1}^{1} |y P^{(\alpha,\beta)}_n(x)|^p \,dydx \leq C_2 \epsilon_n^{p+1} n^{\alpha p-2p-4}.
\end{align*}
It is easy to verify that
\begin{align*}
 \int_{0}^{\eta} \int_{-1}^{1} |y P^{(\alpha,\beta)}_n(x)|^p \,dydx = \frac{2}{p+1} \int_{0}^{\eta} |P^{(\alpha,\beta)}_n(x)|^p \,dx.
\end{align*}
In a similar way as before, we can show that
\begin{align*}
  \int_{0}^{\eta} |P^{(\alpha,\beta)}_n(x)|^p \,dx \leq  \Lambda_1 n^{-p/2} \int_{\sigma}^{ \pi/2}   \theta^{-\alpha p - p/2 +1}  \,d\theta
\end{align*}
for appropriately adjusted constant $\Lambda_1$ and all sufficiently large $n$. Hence
\begin{align}\label{midI}
  \int_{0}^{\eta} \int_{-1}^{1} |y P^{(\alpha,\beta)}_n(x)|^p \,dydx   \leq   \frac { 2\Lambda_1 n^{-p/2}   \sigma^{-\alpha p - p/2 +2}} {(\alpha p + p/2 -2)(p+1)}.
\end{align}
Now let $f(\eta):=w$. For every $\delta >0$ define $h_{\delta}(x):=(1-x)^{I_{conv}(f) + \delta} \frac{w}{(1-\eta)^{I_{conv}(f) + \delta}}$. Then $f(\eta)=h_{\delta}(\eta)$ and $f(1)=h_{\delta}(1)$. By our assumption on $f$ it follows that
\begin{align*}
   (f(x))^{\frac{1} {I_{conv}(f) + \delta } } \geq w^{ \frac{1} {I_{conv}(f) + \delta }  } \frac{1-x}{1-\eta}= (h_{\delta}(x))^{\frac{1} {I_{conv}(f) + \delta } }
\end{align*}
for all $x \in [\eta,1]$. Thus
     \begin{align}\label{lowEn}
       \frac{\epsilon_n}{n^2}=f(x_n) \geq f(1-\frac{1}{4n^2}) \geq h_{\delta}(1-\frac{1}{4n^2}) \geq \frac{w}{(1-\eta)^{I_{conv}(f) + \delta} (2n)^{2I_{conv}(f) + 2\delta } }
     \end{align}
for all sufficiently large $n$.  Now if $\alpha$ is selected so that $2I_{conv}(f)(p+1) +2 - p/2 < \alpha p$, then, by (\ref{midI}) and (\ref{lowEn}), there exists $C_2>0$ such that
\begin{align}\label{left4}
  \int_{0}^{\eta} \int_{-1}^{1} |y P^{(\alpha,\beta)}_n(x)|^p \,dxdy   \leq C_2 \epsilon_n^{p+1} n^{\alpha p-2p-4}
\end{align}
for all  $n \in \mathbb{N}$. Now let $M=Cb^2/2 + C_1 + C_2$ and use the inequalities (\ref{gen}), (\ref{l3}), (\ref{midf}) and (\ref{left4}) to obtain
\begin{align}\label{ubound}
  \iint_K |y P^{(\alpha,\beta)}_n(x)|^p \,dxdy \leq M \epsilon_n^{p+1} n^{\alpha p-2p-4}.
\end{align}
 \indent
By our assumption on $K$ it follows that
\begin{align}\label{lbound}
  \left\|P^{(\alpha,\beta)}_n\right\|^p_{L^p(K)}  \geq \int_{1-\frac{1}{2n^2}}^{1}  \int_{0}^{f(x)} |P^{(\alpha,\beta)}_n(x)|^p \,dydx = \int_{1-\frac{1}{2n^2}}^{1} f(x) |P^{(\alpha,\beta)}_n(x)|^p \,dx.
\end{align}
By making the change of variable $x=1-\frac{z^2}{2n^2}$, we obtain
\begin{align}\label{lbound1}
  \int_{1-\frac{1}{2n^2}}^{1} f(x) |P^{(\alpha,\beta)}_n(x)|^p \,dx = \frac{1}{n^2}\int_{0}^{1}  z f\left(g_n(z)\right) |P^{(\alpha,\beta)}_n\left(g_n(z)\right)|^p  \,dz,
\end{align}
where $g_n(z)=1-\frac{z^2}{2n^2}$. Again certain properties of Jacobi polynomials $P^{(\alpha,\beta)}_n(x)$ play a role. By the formula of Mehler-Heine type (see \cite{Sz}, Theorem 8.1.1.) % Indeed, for appropriately adjusted constant $A$ and  all sufficiently large $n$, we have
\begin{align*}
  \frac{1}{n^2}\int_{0}^{1}  z f\left(g_n(z)\right) |P^{(\alpha,\beta)}_n\left(g_n(z)\right)|^p  \,dz &\geq \\ \frac{n^{\alpha p}}{4^p n^2} \int_{0}^{1} z f(g_n(z)) &(4(z/2)^{-\alpha} J_{\alpha}(z)- 1/\Gamma(\alpha +2))^p \,dz
\end{align*}
for all sufficiently large $n$. Here $J_{\alpha}(z)$ is the Bessel functions of the first kind. Since
\begin{align*}
\min_{z \in [0,1]} \{(z/2)^{-\alpha} J_{\alpha}(z)\} \geq \min_{z \in [0,1]} \left\{ \frac{1}{\Gamma(\alpha +1)}  - \frac{z^2}{4\Gamma(\alpha +2)}  \right\} =  \frac{4 \alpha +3}{4\Gamma(\alpha +2)} ,
\end{align*}
we have
\begin{align}\label{lbound2}
  \frac{1}{n^2}\int_{0}^{1}  z f\left(g_n(z)\right) \left|P^{(\alpha,\beta)}_n\left(g_n(z)\right)\right|^p  \,dz \geq \left(\frac{4\alpha +2 }{4\Gamma(\alpha +2)} \right)^p n^{\alpha p-2} \int_{0}^{1} z f(g_n(z))  \,dz.
\end{align}
Applying integration by parts yields
\begin{align*}
  \int_{0}^{1} z f(g_n(z))  \,dz= \left[\frac{1}{2}z^2 f(g_n(z))\right]^{1}_{0} + \frac{1}{2n^2}\int_{0}^{1} z^3  f'(g_n(z))  \,dz.
\end{align*}
From this and the inequality (\ref{appindex}), it follows that, for all $\delta>0$, it must be that
\begin{align}\label{lbound3}
  (I_{conv}(f) + 1 + \delta)\int_{0}^{1} z f(g_n(z))  \,dz \geq \frac{1}{2} f(g_n(1)) \geq \frac{\epsilon_n}{2n^2}
\end{align}
for all sufficiently large $n$. If $n$ is large enough, then by (\ref{lbound}), (\ref{lbound1}), (\ref{lbound2}) and (\ref{lbound3}) there exists a positive constant $\Upsilon$, independent of $n$, for which
\begin{align}\label{lower}
  \iint_K |P^{(\alpha,\beta)}_n(x)|^p \,dxdy \geq \Upsilon \epsilon_n  n^{\alpha p-4}.
\end{align}
Finally, using the inequalities (\ref{ubound}) and (\ref{lower}), we obtain
\begin{align*}
  \iint_K |P^{(\alpha,\beta)}_n(x)|^p \,dxdy \geq  \frac{ \Upsilon n^{2p}}{ M \epsilon_n^p } \iint_K |yP^{(\alpha,\beta)}_n(x)|^p \,dxdy.
\end{align*}
\end{proof}

As an immediate consequence of Theorem \ref{mainKroo} and Theorem \ref{KrooThm}, we see the following:

\begin{thm}\label{mainKroo2}
  Let $K \subset \mathbb{R}^2$ be a piecewise graph domain. Suppose that one of the strongest cuspidal point of $K$  is regular.  If $\epsilon_n$ is a solution of the equation $2n^2 \omega_K\left( \frac{\epsilon_n}{n^2} \right)=1$, $n \in \mathbb{N}$, then there exists a positive constant $B$, depending on $K$ and on $p$, such that for $Q \in \mathcal{P}_n(\mathbb{R}^d)$, $n \in \mathbb{N}$,
  \begin{align}\label{mainKrooasymp}
    \| \nabla Q \|_{L^p(K)} \leq B \frac{n^2}{\epsilon_n} \| Q \|_{L^p(K)}.
  \end{align}
Moreover, the inequality (\ref{mainKrooasymp})  is asymptotically best possible.
\end{thm}

\section{A growth rate}
Now a similar proof to that of the Theorem \ref{mainKroo} gives the following lemma:
\begin{lem}\label{mainlem}
  Let $\alpha, \beta,  p$ be positive real numbers and $0<\upsilon \leq 1$. Let $f$ be a bounded real-valued function defined on the interval $[0,1]$. Suppose that $f(1)=0$, $f\left( 1- \frac{\upsilon}{n^2}\right)=\frac{\epsilon_n}{n^2}$ and there exists $0<\eta<1$ such that $f|_{[\eta,1]} $ is convex with the property that $I_{conv}(f|_{[\eta,1]}) < \infty$. If $\alpha p \geq 2  I_{conv}(f|_{[\eta,1]}) +2 - p/2$, then
\end{lem}
\begin{align}\label{IiJP}
  \int_{0}^{1} |f(x)| \left|P^{(\alpha,\beta)}_n(x) \right|^p \, dx \sim  \epsilon_n n^{\alpha p -  4}.
\end{align}
It is worth noting that the above lemma provides a refinement and generalization of Theorem 7.34. from \cite{Sz}. \newline \indent
We shall show that, with suitable hypotheses, there is a sort of converse to Theorem \ref{mainKroo}.

\begin{thm}\label{en}
  Let $\{\epsilon_n\}$ be a sequence of real numbers such that $0< \epsilon_{n+1} \leq \epsilon_n$, $\lim_{n \rightarrow \infty} \epsilon_n=0$ and there exist constants $C_n$  with the property that (for all $n$ and $m$)
\begin{align}
 & \frac{\epsilon_n}{n^2} - \frac{\epsilon_m}{m^2} \geq -C_m \frac{\epsilon_m}{2} (1/m^2 - 1/n^2 ), \label{asumpC} \\ &\frac{\epsilon_n}{n^2} - \frac{\epsilon_m}{m^2}= -C_m \frac{ \epsilon_m }{2} (1/m^2 - 1/n^2 ) \Rightarrow C_n \epsilon_n = C_m \epsilon_m, \nonumber \\
 & \sup \{C_n : n \in \mathbb{N}\} < \infty. \label{BofCn}
\end{align}
Then there exist a compact set $D \subset \mathbb{R}^2$, a constant $M>0$ and a sequence of polynomials $P_n \in \mathcal{P}_n(\mathbb{R}^2)$ such that
\begin{align*}
  \frac{ \| \nabla P_n \|_{L^p(D)} } { \| P_n \|_{L^p(D)} }  \sim \frac{n^2}{\epsilon_n} \quad \text{and} \quad \| \nabla Q \|_{L^p(D)} \leq M \frac{n^2}{\epsilon_n}  \| Q \|_{L^p(D)}
\end{align*}
for any $Q \in \mathcal{P}_n(\mathbb{R}^2)$.
\end{thm}
\begin{proof}
  Let $W:=\{1-\frac{1}{2n^2} : n \in \mathbb{N}\} \cup \{1\}$. Define $f(1-\frac{1}{2n^2}):=\frac{ \epsilon_n }{ n^2 }$, $f(1):=0$, $G(1-\frac{1}{2n^2}):= -C_n \epsilon_n $, $G(1):=0$. Using Theorem 1.10 from \cite{AM}, there exists a continuously differentiable function $F: \mathbb{R} \rightarrow \mathbb{R}$ such that $F$ is convex and $F=f$, $F'=G$ on $W$. Now define
\begin{align}\label{function}
  \tilde{f}(x):= \begin{cases}
\epsilon_1 \quad  &\text{for} \quad x \in [0,1/2]\\
F(x) \quad &\text{for} \quad x \in (1/2,1].
\end{cases}
\end{align}
Since $F$ is convex and $F(1)=0$ it follows that $\tilde{f}$ is strictly decreasing on the interval $[1/2,1]$. We shall show that if $\tilde{f}(y)=\tilde{f}(x_1) - \tilde{f}(x_2)$, then
   \begin{align}\label{step}
     \frac{\tilde{f}(y)}{1-y} \leq \frac{ \tilde{f}(x_1) - \tilde{f}(x_2) }{x_2 -x_1}
   \end{align}
for any $1/2 \leq x_1 < x_2\leq 1$ and $1/2 \leq y <1$. Since $\tilde{f}(y)=\tilde{f}(x_1) - \tilde{f}(x_2)$, we have $y \geq x_1$. If $x_1 < x_2 \leq y$, then, by the mean value theorem, there exist $\xi \in (x_1,x_2)$ and $\eta \in (y,1)$ such that
   \begin{align*}
     \frac{-\tilde{f}(y)}{1-y}=\tilde{f}'(\eta), \quad  \frac{ \tilde{f}(x_2) - \tilde{f}(x_1) }{x_2 -x_1}=\tilde{f}'(\xi).
   \end{align*}
Hence, (using the fact that differentiable function of one variable is convex on an interval if and only if its derivative is monotonically non-decreasing on that interval)
  \begin{align*}
     \frac{\tilde{f}(y)}{1-y}=-\tilde{f}'(\eta) \leq  -\tilde{f}'(\xi)= \frac{ \tilde{f}(x_1) - \tilde{f}(x_2) }{x_2 -x_1}.
   \end{align*}
For the case $x_1 \leq y < x_2$, let
\begin{align*}
  S(t,r):=\frac{ F(t) - F(r) }{t -r}.
\end{align*}
It is known that $F$ is convex if and only if $S(t,r)$ is monotonically non-decreasing in $t$, for every fixed $r$. Therefore
  \begin{align*}
     \frac{ F(x_1) - F(x_2) }{x_1 -x_2}=S(x_2,x_1) \leq S(1,x_1)&=\frac{ F(x_1) - F(1) }{x_1 - 1}=S(x_1,1) \\ &\leq  S(y,1)=\frac{ F(y) - F(1) }{y - 1}.
  \end{align*}
Since $F=\tilde{f}$ on the interval $[1/2,1]$,  the inequality (\ref{step}) holds when $x_1 \leq y < x_2$. \newline \indent If we define
   \begin{align}\label{domainD}
     D:=\{(x,y) \in \mathbb{R}^2 : 0 \leq x \leq 1, \, 0 \leq y \leq \tilde{f}(x)\},
   \end{align}
then $D$ is a graph domain with respect to the cylinder $L_{\mathbf{a}}(\epsilon_1/2,\mathbf{u})$, where $\mathbf{a}=(0,\epsilon_1/2)$ and $\mathbf{u}=(0,1)$. From the inequality (\ref{step}) it follows that
   \begin{align}\label{omegaD}
     \omega_D(t)=\sqrt{ \left( 1 - F^{-1}(t)\right)^2 + t^2 }
   \end{align}
whenever  $t \leq \epsilon_1$. Hence
   \begin{align}\label{omegaDn}
     \omega_D\left( \frac{\epsilon_n}{n^2}\right)=\sqrt{  \frac{ 1 }{ 4n^4 } + \left( \frac{ \epsilon_n }{ n^2 } \right)^2 }.
   \end{align}
Let $C:=\sup \{C_n : n \in \mathbb{N}\}$. Select $i, \tau \in \mathbb{N}$ so that $\eta:=\frac{\tau}{i} < 1$ and $1 - \frac{C}{2} + \frac{C}{2} \eta^2 >0$. Now we shall show that if $n \in \mathbb{N}$ is large enough, then there exists $m \in \mathbb{N}$  such that
\begin{align}\label{nm}
  m>n \quad  \text{and} \quad  \omega_D\left( \frac{\epsilon_m}{m^2}\right) < \frac{1}{2n^2}.
\end{align}
For each $n \in \mathbb{N}$ let $l_n \in \mathbb{N}_0$ be such that $n=l_n \tau + s_n$ for some $s_n \in \{0,1,2, \ldots, \tau-1\}$. Define $m_n:=l_n i + n - l_n \tau$. For simplicity of notation, we write $m$ instead of $m_n$. It is clear that
\begin{align}\label{lim}
\lim_{n \rightarrow \infty} \frac{m}{n}=\frac{1}{\eta} >1.
\end{align}
Now choose $n_0$ so large that,
\begin{align*}
  4\epsilon^2_m < \frac{m^4}{n^4} -1
\end{align*}
whenever $n \geq n_0$. Hence $\omega_D\left( \frac{\epsilon_m}{m^2}\right) < \frac{1}{2n^2}$. Thus if $\omega_D\left( \frac{s_n}{n^2}\right)=\frac{1}{2n^2}$, then
\begin{align}\label{sn}
  \frac{\epsilon_m}{m^2} < \frac{s_n}{n^2} < \frac{\epsilon_n}{n^2}.
\end{align}
On the other hand, we may use the inequality (\ref{asumpC}) to write
\begin{align}\label{appC1}
  \frac{\epsilon_n}{n^2}(1 - \frac{C_n}{2} + \frac{C_n}{2} \frac{n^2}{m^2}) \leq \frac{\epsilon_m}{m^2}.
\end{align}
Take $\delta>0$ so that $\xi:=1 - \frac{C}{2} + \frac{C}{2} (\eta^2-\delta) >0$, then there exists $n_1 \in \mathbb{N}$ such that
\begin{align}\label{appC2}
  \frac{\epsilon_n}{n^2} \xi   \leq  \frac{\epsilon_n}{n^2}(1 - \frac{C_n}{2} + \frac{C_n}{2} (\eta^2-\delta) ) \leq \frac{\epsilon_m}{m^2}
\end{align}
whenever $n \geq n_1$.
By (\ref{sn}), (\ref{appC2}) and Theorem \ref{KrooThm} there exists a constant $B>0$ such that
  \begin{align}\label{corKroo}
    \| \nabla Q \|_{L^p(D)} \leq B \frac{n^2}{\epsilon_n}  \| Q \|_{L^p(D)}
  \end{align}
for any $Q \in \mathcal{P}_n(\mathbb{R}^2)$. Now let $\nu:=\max \{n_0, n_1, r\}$. If $x \in [1- \frac{1}{2\nu^2},1)$, then there exists $\varsigma \in \mathbb{N}$ such that
\begin{align*}
  x \in [1-\frac{1}{2\varsigma^2}, 1-\frac{1}{2(\varsigma + 1)^2}] \subset [1-\frac{1}{2\varsigma^2}, 1-\frac{1}{2m_{\varsigma}^2}].
\end{align*}
Hence, by properties of $F$,
\begin{align}\label{new}
   \frac{C}{2 \xi} F(x) \geq \frac{C}{2 \xi} F(1-\frac{1}{2m_\varsigma^2})=&\frac{C \epsilon_{m_\varsigma} }{2 \xi m_\varsigma^2} \geq \frac{C \epsilon_\varsigma}{2 \varsigma^2} \nonumber \\ &\geq \frac{C_\varsigma \epsilon_\varsigma}{2 \varsigma^2} = -F'(1-\frac{1}{2\varsigma^2}) \frac{1}{2\varsigma^2}\geq -F'(x)(1-x).
\end{align}
  Now, if $P_n(x,y):=yP^{(\alpha,\beta)}_n(x)$, then
\begin{align}\label{roz}
  \left\| P_n \right\|^p_{L^p(D)}=& \int_{0}^{1- \frac{1}{2\nu^2}} \int_{0}^{\tilde{f}(x)} |y P^{(\alpha,\beta)}_n(x)|^p \,dydx \nonumber \\ &+ \int_{ 1- \frac{1}{2\nu^2} }^{1- \frac{1}{2n^2}}  \int_{0}^{\tilde{f}(x)} |y P^{(\alpha,\beta)}_n(x)|^p \,dydx \nonumber \\ &+ \int_{1- \frac{1}{2n^2}}^{1}  \int_{0}^{\tilde{f}(x)} |y P^{(\alpha,\beta)}_n(x)|^p \,dydx
\end{align}
  for $n>\nu$. It is easy to conclude that
\begin{align}\label{war1}
  \int_{1- \frac{1}{2n^2}}^{1}  \int_{0}^{\tilde{f}(x)} |y P^{(\alpha,\beta)}_n(x)|^p \,dydx \leq  \frac{2}{p+1} \left(\frac{\epsilon_n}{n^2} \right)^{p+1}  \int_{1- \frac{1}{2n^2}}^{1}  |P^{\alpha}_n(x)|^p \,dx
\end{align}
An argument similar to the one we gave for $\int_{x_n}^{1}  |P^{(\alpha\, \beta)}_n(x)|^p \,dx$ shows that there exists $\vartheta>0$ such that
\begin{align}\label{war2}
  \int_{1- \frac{1}{2n^2}}^{1}  |P^{(\alpha,\beta)}_n(x)|^p \,dx \leq \vartheta n^{\alpha p -2}
\end{align}
for all sufficiently large $n$. Thus by (\ref{war1}) and (\ref{war2}),
\begin{align}\label{war3}
  \int_{1- \frac{1}{2n^2}}^{1}  \int_{0}^{\tilde{f}(x)} |y P^{(\alpha,\beta)}_n(x)|^p \,dydx \leq \frac{2\vartheta}{p+1}  \epsilon_n^{p+1} n^{\alpha p -2p-4}.
\end{align}
Using the methods similar to ones used in the proof of Theorem \ref{mainKroo}, applying (\ref{new}) instead of (\ref{appindex}), we have
\begin{align}
  &\upsilon \epsilon_n  n^{\alpha p-4} \leq \int_{1- \frac{1}{2n^2}}^{1}  \int_{0}^{\tilde{f}(x)} |P^{(\alpha\, \beta)}_n(x)|^p \,dxdy , \label{lowernew1} \\
  &\int_{ 1- \frac{1}{2\nu^2} }^{1- \frac{1}{2n^2}}  \int_{0}^{\tilde{f}(x)} |y P^{(\alpha,\beta)}_n(x)|^p \,dydx \leq \vartheta_1 \epsilon_n^{p+1} n^{\alpha p -2p-4} \label{lowernew2}
\end{align}
for $\alpha p + p/2 -2 > (p+1)\frac{C}{ \xi}$,  appropriately adjusted constants $\upsilon, \vartheta_1$ and all sufficiently large $n$. \newline \indent Now we shall show that there exist $\vartheta_2>0$, $n_2,k \in \mathbb{N}$ such that
\begin{align}\label{cha0}
 \vartheta_2 n^{-p/2} \leq \int_{ 1- \frac{1}{2\nu^2} }^{1- \frac{1}{2k^2}}  \int_{0}^{\tilde{f}(x)} |y P^{(\alpha,\beta)}_n(x)|^p \,dydx
\end{align}
for all $n \geq n_2$. By properties of $f$, we may write
\begin{align}\label{cha}
  \frac{1}{p+1} (\tilde{f}(1- \frac{1}{2\nu^2}))^{p+1} \int_{ 1- \frac{1}{2\nu^2} }^{1- \frac{1}{2k^2}} |P^{(\alpha,\beta)}_n(x)|^p \,dx   \leq \int_{ 1- \frac{1}{2\nu^2} }^{1- \frac{1}{2k^2}}  \int_{0}^{\tilde{f}(x)} |y P^{(\alpha,\beta)}_n(x)|^p \,dx.
\end{align}
The change of variable $x=\cos \theta$, give us
\begin{align}\label{cha1}
  \int_{ 1- \frac{1}{2\nu^2} }^{1- \frac{1}{2k^2}} |P^{(\alpha,\beta)}_n(x)|^p \,dx = \int_{u_k}^{u_{\nu}} |P^{(\alpha,\beta)}_n(\cos \theta)|^p \sin \theta \,d \theta,
\end{align}
where $u_k=\arccos (1- \frac{1}{2k^2})$ and $u_{\nu}=\arccos (1- \frac{1}{2\nu^2})$. From Theorem 8.21.8 in \cite{Sz}
\begin{align}\label{cha2}
  &P^{(\alpha,\beta)}_n(\cos \theta)=n^{-1/2} k(\theta) \cos(N \theta + \gamma) + O(n^{-3/2}), \\
  &k(\theta)= \pi^{-1/2} (\sin (\theta/2))^{-\alpha -1/2} (\cos (\theta/2))^{-\beta -1/2}, \quad N=n+(\alpha + \beta +1)/2,  \nonumber \\
  &\gamma=-(\alpha + 1/2) \pi /2, \quad 0< \theta < \pi. \nonumber
\end{align}
Moreover, if $\delta >0$, then the bound for the error term holds uniformly in the interval $[\delta, \pi - \delta]$. Let $0 < \theta_1 < \theta_2 < \ldots < \theta_n < \pi$ be the zeros of $P^{(\alpha,\beta)}_n(\cos \theta)$. Then, by (8.9.8) of \cite{Sz}, the zeros $\theta_l$ from a fixed interval $[a, b]$ in the interior of $[0, \pi]$ can be written in the following form
\begin{align}\label{cha3}
  \theta_l= N^{-1} ((l-1/2)\pi - \gamma + K \pi + \xi_n),
\end{align}
where $K$ is a fixed integer (depending only on $\alpha, \beta, a, b$) and $\xi_n \rightarrow 0$. If $k > \nu$, $0<2\rho < \frac{1}{2\nu^2} - \frac{1}{2k^2}$, then
\begin{align}\label{cha4}
  0 < \int_{U} \,dx,
\end{align}
where $$U:= [1-\frac{1}{2\nu^2}, 1- \frac{1}{2k^2} ] \setminus  \bigcup_{l=1}^n \left( \frac{ (l-1/2)\pi - \gamma + K \pi - \rho}{N}, \frac{ (l-1/2)\pi - \gamma + K \pi + \rho} {N} \right).$$
By (\ref{cha1}), (\ref{cha2}), (\ref{cha3}) and (\ref{cha4}), we have
\begin{align}\label{cha5}
  \int_{ 1- \frac{1}{2\nu^2} }^{1- \frac{1}{2k^2}} |P^{(\alpha,\beta)}_n(x)|^p \,dx \sim n^{-p/2}.
\end{align}
Combining (\ref{cha}) and (\ref{cha5}) we obtain (\ref{cha0}).
\newline \indent By properties of $\tilde{f}$ it follows that
\begin{align}\label{ep}
   \int^{1- \frac{1}{2\nu^2}}_{0}  \int_{0}^{\tilde{f}(x)} |y P^{(\alpha,\beta)}_n(x)|^p \,dydx \leq& \frac{\epsilon_1^{p+1}}{p+1} \int^{1- \frac{1}{2\nu^2}}_{0}  |P^{(\alpha,\beta)}_n(x)|^p \,dx \nonumber \\ &= \frac{\epsilon_1^{p+1}}{p+1} \int_{u_\nu}^{\pi/2} |P^{(\alpha,\beta)}_n( \cos \theta)|^p \sin \theta \, d\theta.
\end{align}
Hence, by (\ref{cha2}),
\begin{align}\label{ep1}
   \int^{1- \frac{1}{2\nu^2}}_{0}  \int_{0}^{\tilde{f}(x)} |y P^{(\alpha,\beta)}_n(x)|^p \,dydx \leq  \vartheta_3 n^{-p/2} \int_{u_\nu}^{\pi/2}  \theta^{-\alpha p - p/2 +1} \,d\theta
\end{align}
for appropriately adjusted constant $\vartheta_3$. Together with (\ref{lowernew2}) and (\ref{cha0}), this last estimate implies that there exists a constant $\vartheta_4>0$ so that
\begin{align}\label{ep2}
  \int^{1- \frac{1}{2\nu^2}}_{0}  \int_{0}^{\tilde{f}(x)} |y P^{(\alpha,\beta)}_n(x)|^p \,dydx \leq  \vartheta_4 \epsilon_n^{p+1} n^{\alpha p -2p-4}
\end{align}
whenever $\alpha p + p/2 -2 > (p+1)\frac{C}{ \xi}$. Putting together (\ref{roz}), (\ref{war3}), (\ref{lowernew1}), (\ref{lowernew2}) and (\ref{ep2}), we find that
\begin{align}\label{ep3}
  \frac{ \left\| \frac{ \partial P_n}{\partial y}  \right\|_{L^p(D)} } { \| P_n \|_{L^p(D)} } \geq \vartheta_5 \frac{n^2}{\epsilon_n},
\end{align}
where $\vartheta_5 > 0$ is a constant independent of $n$.
Finally, (\ref{corKroo}) and (\ref{ep3}) yield that
   \begin{align*}
     \frac{ \| \nabla P_n \|_{L^p(D)} } { \| P_n \|_{L^p(D)} }  \sim \frac{n^2}{\epsilon_n}.
   \end{align*}
\end{proof}

\begin{lem}\label{cor}
  Define a function $\varphi_{\iota}$ on the interval $[0,1]$ as follows:
  \begin{align*}
    \varphi(t)= \begin{cases}
                          \frac{t}{1+ \ln(1/t) }, & \mbox{if } \, t \in (0,1] \\
                          0, & \mbox{for } \,  t=0.
                        \end{cases}
  \end{align*}
  Let $E_\iota=\{(x,y) \in \mathbb{R}^2 : 0 \leq x \leq 1,  0 \leq  y \leq \varphi ((1-x)^\iota) \}$. Then  there exist a positive constant $B_\iota$ and a sequence of polynomials $P_n$ such that
   \begin{align}\label{setE}
      &\frac{ \| \nabla P_n \|_{L^p(D)} } { \| P_n \|_{L^p(D)} }  \sim n^{2\iota}(1+\iota \ln(2n^2)) \quad \text{and} \\  \| \nabla &Q\|_{L^{p}(E_\iota)} \leq   B_\iota n^{2\iota}(1+\iota \ln(2n^2)) \left\| Q \right\|_{L^{p}(E_\iota)}
       \end{align}
for any $Q \in \mathcal{P}_n(\mathbb{R}^2)$.
\end{lem}
We omit the details of the proof of Lemma \ref{cor}, as they would repeat ideas that we presented  in the proof of Theorem \ref{en}. Using (\ref{setE}), one can see that $\mu_p(E_\iota)=2\iota$  and $L^p$ Markov type inequality on $E_\iota$ does not hold with the exponent $\mu_p(E_\iota)=2\iota$. This generalizes Proposition 2.6 of \cite{BLM}.

\section*{Acknowledgment}
The author was supported by the Polish National Science Centre (NCN) Opus grant no. 2017/25/B/ST1/00906.


%% References


%%\bibliographystyle{elsarticle-num}
%%\bibliography{<your-bib-database>}


{\footnotesize

\noindent Tomasz Beberok\\
Department of Applied Mathematics,\\
University of Agriculture in Krakow,\\
ul. Balicka 253c, 30-198 Kraków, Poland\\
email: tomasz.beberok@urk.edu.pl
}
% ------------------------------------------------------------------------
\end{document}